\documentclass[12pt,reqno]{article}

\usepackage{amssymb,amsmath,amsthm,amsfonts,color}
\usepackage[colorlinks=true,urlcolor=blue,linkcolor=blue,citecolor=blue,pdfstartview=FitH]{hyperref}
\usepackage{authblk}
\usepackage[sort,comma]{natbib}

\newtheorem{thm}{Theorem}[section]
\newtheorem{cor}[thm]{Corollary}

\newtheorem{prop}[thm]{Proposition}
\newtheorem{lem}[thm]{Lemma}
\theoremstyle{definition}
\newtheorem{defn}[thm]{Definition}

\newcommand{\lemref}[1]{Lemma~{\rm \ref{#1}}}

\newcommand{\propref}[1]{Proposition~{\rm \ref{#1}}}
\newcommand{\defref}[1]{Definition~{\rm \ref{#1}}}
\newcommand{\remref}[1]{Remark~{\rm \ref{#1}}}

\newtheorem{rem}[thm]{Remark}

\makeatletter \@addtoreset{equation}{section}

\allowdisplaybreaks

\newcommand{\EE}{\mathbb E}
\newcommand{\NN}{\mathbb N}
\newcommand{\RR}{\mathbb R}
\newcommand{\PP}{\mathbb P}

\newcommand{\wdt}{\widetilde}

\title{\Large \sc A Counterintuitive Example in Inventory Management}
\author{Kurt L. Helmes, Richard H. Stockbridge and Chao Zhu}
\date{}
\begin{document}
\maketitle

\vspace{5 mm}

\begin{abstract}
  The paper \cite{helm:17} studies an inventory management problem under a long-term average cost criterion using weak convergence methods applied to average expected occupation and average expected ordering measures.  Under the natural condition of inf-compactness of the holding cost rate function, the average expected occupation measures are seen to be tight and hence have weak limits.  However inf-compactness is not a natural assumption to impose on the ordering cost function.  For example, the cost function composed of a fixed cost plus proportional (to the size of the order) cost is not inf-compact.  Intuitively, it would seem that imposing a requirement that the long-term average cost be finite ought to imply tightness of the average expected ordering measures; a lack of tightness should mean that the inventory process spends large amounts of time in regions that have arbitrarily large holding costs resulting in an infinite long-term average cost.  This paper demonstrates that this intuition is incorrect by identifying a model and an ordering policy for which the resulting inventory process has a finite long-term average cost, tightness of the average expected occupation measures but which lacks tightness of the average expected ordering measures.
\end{abstract}

\vspace{3 mm}

\vspace{5 mm}

\section{Introduction}
This paper fleshes out a remark made in \cite{helm:17}.  That paper establishes under very weak conditions the existence of an optimal $(s_*,S_*)$ ordering policy for inventory management in one-dimensional diffusion models.  The problem under consideration involves minimizing the long-term average expected cost of an inventory process $X$ taking values in some interval ${\cal E} \subseteq \RR$ satisfying
\begin{equation} \label{controlled-dyn}
X(t) = x_0 + \int_0^t \mu(X(s))\, ds + \int_0^t \sigma(X(s))\, dW(s) + \sum_{k=1}^\infty I_{\{\tau_k \leq t\}} Y_k, \qquad t\geq 0,
\end{equation}
in which $x_0 \in {\cal E}$ denotes the initial inventory level, $(\tau,Y) = \{(\tau_k,Y_k): k \in \NN\}$ is an ordering policy in which $\tau_k$ denotes the time of the $k^{\mbox{\footnotesize th}}$ order and $Y_k \geq 0$ is the size of the order.  The drift rate given by $\mu$ and the diffusion coefficient $\sigma$ are assumed to be continuous, with $\sigma$ non-degenerate; $W$ is a standard Brownian motion process.  Note the condition $Y_k \geq 0$ means that this model only allows the manager to increase the inventory through ordering; it does not provide him or her the ability to decrease inventory.  Define the set ${\cal R} = \{(y,z) \in {\cal E}^2: y < z\}$ in which $y$ represents the pre-order inventory level and $z$ is the post-order level and set $\overline{\cal R} = \{(y,z) \in {\cal E}^2: y \leq z\}$.  The difference between ${\cal R}$ and $\overline{\cal R}$ is that $\overline{\cal R}$ includes orders of size $0$ whereas ${\cal R}$ does not.

The long-term average criterion is
\begin{equation} \label{lta-cost}
J_0(\tau,Y) = \limsup_{t\rightarrow \infty} \mbox{$\frac{1}{t}$} \EE\left[\int_0^t c_0(X(s))\, ds + \sum_{k=1}^\infty I_{\{\tau_k \leq t\}} c_1(X(\tau_k-),X(\tau_k))\right]
\end{equation}
in which $c_0\geq 0$ is the holding/back-order cost rate function and $c_1 \geq k_1 > 0$ denotes the cost for ordering.  Here, $k_1$ represents the fixed cost per order while $c_1(y,z)$ gives the total ordering cost for raising the inventory level from $y$ to $z$.  Notice, in particular, the assumption $c_1 \geq k_1$ means that actively ordering nothing incurs a cost of at least $k_1$ so differs from not ordering which has no cost.  The paper \cite{helm:17} also imposes minor additional conditions such as continuity on $c_0$ and $c_1$ in Condition 2.3.

The approach taken in \cite{helm:17} is to capture the expected behaviour of the inventory process $X$ associated with an ordering policy $(\tau,Y)$ and its costs through average expected occupation and average expected ordering measures.  These are defined for each $t > 0$ by
\begin{equation} \label{mus-t-def}
\begin{array}{rcll}
\mu_{0,t}(\Gamma_0) &=& \displaystyle \mbox{$\frac{1}{t}$} \EE\left[\int_0^t I_{\Gamma_0}(X(s))\, ds\right], & \quad \Gamma_0 \in {\cal B}({\cal E}), \rule[-15pt]{0pt}{15pt} \\
\mu_{1,t}(\Gamma_1)  &=& \displaystyle \mbox{$\frac{1}{t}$} \EE\left[\sum_{k=1}^\infty I_{\{\tau_k \leq t\}} I_{\Gamma_1}(X(\tau_k-),X(\tau_k))\right], & \quad \Gamma_1 \in {\cal B}(\overline{\cal R}).
\end{array}
\end{equation}
The long-term average cost is expressed using $\mu_{0,t}$ and $\mu_{1,t}$ as
$$J_0(\tau,Y) = \limsup_{t\rightarrow \infty} \left(\int_{\cal E} c_0(x)\, \mu_{0,t}(dx) + \int_{\overline{\cal R}} c_1(y,z)\, \mu_{1,t}(dy\times dz)\right).$$
The paper then employs weak convergence techniques to more easily show optimality of an $(s,S)$ policy in the general class of admissible policies.  

One of the benefits of this approach is the simplicity of establishing tightness of $\{\mu_{0,t}\}$ as $t\rightarrow \infty$.  For example, when ${\cal E}$ is a bounded interval, viewing each $\mu_{0,t}$ as a measure on the closure $\overline{\cal E}$ of ${\cal E}$ immediately implies that $\{\mu_{0,t}\}$ is tight so there will exist limiting measures as $t\rightarrow \infty$.  When ${\cal E}$ is unbounded, it is natural to assume that $c_0$ converges to $\infty$ at each infinite boundary; this means that ``an infinite holding cost rate occurs when an infinite amount of inventory is either present or back-ordered.''  Under this assumption of inf-compactness on $c_0$, Proposition~3.3 of \cite{helm:17} shows that $\{\mu_{0,t}\}$ is tight whenever $J_0(\tau,Y) < \infty$; in this paper, \propref{mu0t-tight} establishes the same result for our particular model and ordering policy.

Imposing inf-compactness on $c_1$ would similarly imply $\{\mu_{1,t}\}$ is tight whenever $J_0(\tau,Y) < \infty$.  However, the assumption of inf-compactness is not natural for the ordering cost function $c_1$.  For example, the most commonly studied ordering cost function is $c_1(y,z) = k_1 + k_2(z-y)$ for $(y,z) \in {\cal R}$ in which the fixed cost of $k_1$ is charged along with a cost that is proportional to the size of the order.  When ${\cal E} = \RR$, for example, the set $\{(y,z) \in {\cal R}: c_1(y,z) \leq K\}$ is not compact for any $K < \infty$.  In addition, if one were to allow efficiencies of scale by considering a concave function $H: \RR_+ \rightarrow \RR_+$ and taking $c_1(y,z) = k_1 + H(z-y)$, again the set $\{(y,z): c_1(y,z) \leq K\}$ is not compact for any $K < \infty$. 

It would seem that a sufficient condition to ensure $\{\mu_{1,t}\}$ is tight ought to be that $J_0(\tau,Y) < \infty$.  Intuitively, placing many orders resulting in unbounded inventory levels means that $X$ ought to spend large amounts of time in regions where the function $c_0$ imposes a heavy penalty leading to an unbounded expected holding cost and hence $J_0(\tau,Y) = \infty$.

The purpose of this paper is to illustrate that this intuition is incorrect.  To illustrate the idea simply, we begin by considering a deterministic inventory model having constant demand rate and examine a particular ordering policy.  The analysis is fairly straightforward.  The second example is very similar but perturbs the deterministic demand by a Brownian motion process.  The ensuing drifted Brownian motion process requires more sophisticated analysis.

\section{Deterministic Demand Rate Example}
The aim of this example is to illustrate that the tightness of the average ordering measures cannot be easily characterized for the most commonly studied ordering cost function. The example consists of a deterministic inventory model so to distinguish this model from the stochastic models in the rest of the paper, lower case notation will be used.  

An ordering policy $(\mathbf{t},\mathbf{y})$ consists of a sequence of times $\mathbf{t}=\{t_k: k\in \NN\}$ and corresponding sequence of order sizes $\mathbf{y}=\{y_k: k\in \NN\}$.  The inventory process $x$ taking values in $\RR$ satisfies
\begin{equation} \label{determ-dyn}
x(t) = -t + \sum_{k=1}^\infty I_{\{t_k \leq t\}} y_k;
\end{equation}
notice that the initial inventory is $x(0)=0$ and demand is at a constant rate of $1$ unit per unit of time.  The cost structure consists of the holding/back-order rate function $c_0(x) = 2|x|$, the ordering cost function $c_1(y,z) = k_1 + (z-y)$ and the long-term average cost is
\begin{equation} \label{determ-lta-cost}
j_0(\mathbf{t},\mathbf{y}) = \limsup_{t\rightarrow \infty} \mbox{$\frac{1}{t}$} \left(\int_0^t c_0(x(s))\, ds + \sum_{k=1}^\infty I_{\{t_k \leq t\}} c_1(x(t_k-),x(t_k))\right).
\end{equation}
For each $t > 0$, the average occupation and and average ordering measures, $m_{0,t}$ and $m_{1,t}$ respectively, are defined by
\begin{equation} \label{determ-mus-t-def}
\begin{array}{rcll}
m_{0,t}(\Gamma_0) &=& \displaystyle \mbox{$\frac{1}{t}$} \int_0^t I_{\Gamma_0}(x(s))\, ds, & \quad \Gamma_0 \in {\cal B}({\cal E}), \rule[-15pt]{0pt}{15pt} \\
m_{1,t}(\Gamma_1) &=& \displaystyle \mbox{$\frac{1}{t}$} \sum_{k=1}^\infty I_{\{t_k \leq t\}} I_{\Gamma_1}(x(t_k-),x(t_k)), & \quad \Gamma_1 \in {\cal B}(\overline{\cal R}).
\end{array}
\end{equation}

The goal of this example is to identify a particular ordering policy $(\mathbf{t},\mathbf{y})$ with $j_0(\mathbf{t},\mathbf{y}) < \infty$, $\{m_{0,t}\}$ tight as $t\rightarrow \infty$, such that  the collection $\{m_{1,t}\}$ is not tight; that is, show that for some $\epsilon > 0$ and for any compact set $\Gamma \subset {\cal R}$, there is some $t$ such that $m_{1,t}(\Gamma^c) > \epsilon$.

We begin with a description of the ordering policy $(\mathbf{t},\mathbf{y})$ which runs in cycles, each of which is composed of two phases.  For cycle $i=1,2,3,\ldots$, Phase 1 consists of using the $(0,1)$-ordering policy a total of $2^{i-1}$ times.  Phase 2 is quite non-standard involving  a single $(0,2^{(i-1)/2})$-ordering policy followed immediately by using the $(2^{(i-1)/2},2^{(i-1)/2})$-ordering policy $2^{i-1}$ times; this phase consists of placing many orders of size $0$ in a row that incur the fixed cost but no additional cost in addition to the single order of higher amount.

Rather than use a single index for the order times and amounts, we employ a double index $(i,j)$ in which $i$ denotes the cycle and $j$ gives the number of the order within cycle $i$.  

\begin{defn}[{\em The Policy} $(\mathbf{t},\mathbf{y})$] \label{determ-ty-def}
For the first cycle, define
$$\left\{\begin{array}{rcl} t_{1,1}&=&0, \\ y_{1,1}&=&1, \end{array}\right. \quad \mbox{and} \quad \left\{\begin{array}{rcl} t_{1,2}&=&1, \\ y_{1,2}&=&1, \end{array} \right. \quad \left\{\begin{array}{rcl} t_{1,3}&=&t_{1,2}=1 \\ y_{1,3}&=&0.\end{array} \right.$$
Observe that the first order occurs immediately, making $x(0+) = 1$, the inventory level then declines at rate 1 until it hits $\{0\}$ at time $1$ completing the first phase.  Phase 2 begins with an order up to level $2^{(1-1)/2}=1$ at time $1$, immediately followed at time $1$ by an order of size $0$; the process $x$ then decreases and hits $0$ at time $2$, completing Phase 2 of the cycle 1.

For cycle $i=2,3,4,\ldots$, define
$$\left\{\begin{array}{rcl} t_{i,1}&=&t_{i-1,2^i-1} + 2^{(i-1)/2}, \\ y_{i,1}&=&1, \end{array}\right. \quad \mbox{and}\quad \left\{\begin{array}{rcl} t_{i,j}&=& t_{i,j-1}+1,\\ y_{i,j}&=& 1, \end{array} \right. \mbox{ for } j=2, \ldots, 2^{i-1},$$
completing Phase 1 of cycle $i$.  Then define the orders during Phase 2 by 
$$\left\{\begin{array}{rcl} t_{i,2^{i-1}+1}&=& t_{i,2^{i-1}}+1, \\ y_{i,2^{i-1}+1}&=&2^{(i-1)/2}, \end{array} \right. \quad \mbox{and} \quad \left\{\begin{array}{rcl} t_{i,2^{i-1}+1+j}&=&t_{i,2^{i-1}+1}, \\ y_{i,2^{i-1}+1+j}&=&0, \end{array} \right. \mbox{ for } j=1,\ldots,2^{i-1};$$
cycle $i$ ends when $x$ hits $\{0\}$ at time $t_{i+1,1}$.  Notice that as $i$ increases, the {\em number}\/ of $(0,1)$-ordering sub-cycles in Phase 1 increases exponentially by a factor of 2, followed in Phase 2 by a single order of (base $\sqrt{2}$) exponentially increasing {\em size}\/ which is immediately followed by a (base $2$) exponentially increasing {\em number}\/ of orders of size $0$. 
\end{defn}

\begin{rem} \label{no-0-size-order}
When formulating the ordering costs, traditionally no distinction is made between not ordering and ordering nothing, with no cost incurred in either case.  In contrast, the formulation in this paper has orders of size $0$ incur the fixed cost $k_1$.  The policy $(\mathbf{t},\mathbf{y})$ uses $0$-size orders to create non-trivial masses in the average ordering measure $m_{1,t}$ at arbitrarily large values on the diagonal $z=y$ without affecting the length of Phase 2, provided $t$ is sufficiently large.  Under the traditional formulation, it is possible to place many orders of suitably small sizes, resulting in similar masses in neighbourhoods near the diagonal at arbitrary distances from the origin (with $t$ large), such that the length of Phase 2 is barely increased.  The analysis is essentially the same as in this manuscript but requires more careful bookkeeping without affecting the limiting results.  The costly $0$-size orders considerably simplify the computations.
\end{rem}

\begin{prop}
The long-term average cost of the policy $(\mathbf{t},\mathbf{y})$ of \defref{determ-ty-def} satisfies
$$j_0(\mathbf{t},\mathbf{y}) < \infty.$$
\end{prop}

\begin{proof}
It is helpful to determine both the length of each cycle $i$ and the associated cost of cycle $i$.  Note that Phase 1 consists of $2^{i-1}$ $(0,1)$-sub-cycles, each of duration $1$ unit of time.  The length of Phase 2 is determined by the single $(0,2^{(i-1)/2})$-policy lasting $2^{(i-1)/2}$ units of time.  Thus, the length of cycle $i$ is
$$t_{i+1,1} -t_{i,1} = 2^{i-1} + 2^{(i-1)/2}.$$
In examining the costs incurred during cycle $i$, consider first the holding costs.  During Phase 1 by construction, $x(t) = 1-t+t_{i,j}$ for $t$ in the unit time interval $(t_{i,j},t_{i,j+1})$ for $j=1,\ldots,2^{i-1}$.  Phase 2 occupies the interval $(t_{i,2^{i-1}+1},t_{i+1,1})=(t_{i,2^{i-1}+1},t_{i,2^{i-1}+1}+2^{(i-1)/2})$ with $x(t) = 2^{(i-1)/2} - t + t_{i,2^{i-1}+1}$.  Thus, the holding cost during cycle $i$ is
\begin{eqnarray} \label{determ-cycle-i-holding-cost} \nonumber
\int_{t_{i,1}}^{t_{i+1,1}} c_0(x(t))\, dt &=& \displaystyle \sum_{j=1}^{2^{i-1}} \int_{t_{i,j}}^{t_{i,j+1}} 2(1-t+t_{i,j})\, dt + \int_{t_{i,2^{i-1}+1}}^{t_{i+1,1}} 2(2^{(i-1)/2}-t+t_{i,2^{i-1}+1})\, dt \\
&=& 2^{i-1} + 2^{i-1} = 2^i. \rule{0pt}{18pt}
\end{eqnarray}
During cycle $i$, there are $2^{i-1}$ orders of size 1, one order of size $2^{(i-1)/2}$ and $2^{i-1}$ orders of size $0$, resulting in a total ordering cost of 
\begin{equation} \label{determ-cycle-i-order-cost}
\begin{array}{rcl} \displaystyle
\sum_{j=1}^{2^i+1} c_1(x(t_{i,j}-),x(t_{i,j})) &=& (k_1+1) 2^{i-1} + (k_1 + 2^{(i-1)/2}) + k_1 2^{i-1} \\
&=& (2k_1+1) 2^{i-1} + 2^{(i-1)/2} + k_1.
\end{array}
\end{equation}
Obviously, the total cost during cycle $i$ is the sum of \eqref{determ-cycle-i-holding-cost} and \eqref{determ-cycle-i-order-cost}.

It will be helpful to determine the total elapsed time during the first $n$ cycles for $n=1,2,3,\ldots$ as well as the total cost during these cycles.  The total time is
$$t_{n+1,1} = \sum_{i=1}^n (2^{i-1} + 2^{(i-1)/2}) = 2^n + \frac{2^{n/2}-1}{2^{1/2}-1} - 1$$
while the total cost of the first $n$ cycles is 
\begin{eqnarray*}
\lefteqn{\int_0^{t_{n+1,1}} c_0(x(t))\, dt + \sum_{i=1}^n \sum_{j=1}^{2^i+1} c_1(x(t_{i,j}-),x(t_{i,j}))} \\ 
&\qquad =& \sum_{i=1}^n (2^i + (2k_1+1) 2^{i-1} + 2^{(i-1)/2} + k_1) \\
&\qquad =& (k_1+1)2^{n+1} + 2^n + \frac{2^{n/2}-1}{2^{1/2}-1} + k_1n - 2k_1 - 3.
\end{eqnarray*}

We now turn to an examination of the average cost over the interval $[0,t]$.  For each $t \geq 0$, define $i(t)$ to satisfy $t_{i(t),1} \leq t < t_{i(t)+1,1}$; that is, the time $t$ occurs during the $i(t)^{\mbox{\footnotesize th}}$ cycle.  Then
\begin{eqnarray} \label{determ-avg-cost-bd} \nonumber
\lefteqn{\frac{1}{t} \left(\int_0^t c_0(x(s))\, ds + \sum_{i=1}^{\infty}\sum_{j=1}^{2^i+2} I_{\{t_{i,j} \leq t\}} c_1(x(t_{i,j}-),x(t_{i,j}))\right)} \\  \nonumber
&\qquad \leq& \frac{1}{t} \left(\int_0^{t_{i(t)+1,1}} c_0(x(s))\, ds + \sum_{i=1}^{\infty}\sum_{j=1}^{2^i+2} I_{\{t_{i,j} \leq t_{i(t)+1,1}\}} c_1(x(t_{i,j}-),x(t_{i,j}))\right) \\ 
&\qquad =& \frac{t_{i(t)+1,1}}{t} \cdot \frac{1}{t_{i(t)+1,1}} \sum_{i=1}^{i(t)+1} (2^i + (2k_1+1) 2^{i-1} + 2^{(i-1)/2} + k_1) \\ \nonumber
&\qquad \leq& \frac{t_{i(t)+1,1}}{t_{i(t),1}} \cdot \frac{(k_1+1)2^{i(t)+2} + 2^{i(t)+1} + \frac{2^{(i(t)+1)/2}-1}{2^{1/2}-1} + k_1 (i(t)+1) - 2k_1 - 3}{2^{i(t)+1} + \frac{2^{(i(t)+1)/2}-1}{2^{1/2}-1} - 1} \\ \nonumber
&\qquad \leq& \frac{t_{i(t)+1,1}}{t_{i(t),1}} \cdot \frac{2k_1+3 + \frac{2^{-(i(t)+1)/2}-2^{-(i(t)+1)}}{2^{1/2}-1} + (k_1i(t) - k_1 - 3)2^{-(i(t)+1)}}{1 + \frac{2^{-(i(t)+1)/2}-2^{-(i(t)+1)}}{2^{1/2}-1} - 2^{-(i(t)+1)}} .
\end{eqnarray}
Using the expression for $t_{n,1}$ yields
$$\frac{t_{i(t)+1,1}}{t_{i(t),1}} = \frac{2^{i(t)+1} + \frac{2^{(i(t)+1)/2}-1}{2^{1/2}-1} - 1}{2^{i(t)} + \frac{2^{i(t)/2}-1}{2^{1/2}-1} - 1} = \frac{2 + \frac{2^{(1-i(t))/2}-2^{-i(t)}}{2^{1/2}-1} - 2^{-i(t)}}{1 + \frac{2^{-i(t)/2}-1}{2^{1/2}-1} - 2^{-i(t)}}$$
so $\lim_{t\rightarrow \infty} \frac{t_{i(t),1}}{t_{i(t)-1,1}} = 2$.  As a result, there exists some $T < \infty$ such that for all $t \geq T$, $\frac{t_{i(t)+1,1}}{t_{i(t),1}} \leq 3$ and therefore 
$$\limsup_{t\rightarrow \infty} \frac{1}{t} \left(\int_0^t c_0(x(s))\, ds + \sum_{i=1}^\infty \sum_{j=1}^{2^i+2} I_{\{t_{i,j} \leq t\}} c_1(x(t_{i,j}-),x(t_{i,j}))\right) \leq 6k_1 + 9 < \infty,$$
completing the proof.
\end{proof}

Since the holding cost rate function $c_0$ is inf-compact and $j_0(\mathbf{t},\mathbf{y}) < \infty$, the next result is straightforward; we provide the proof of the result for the stochastic example in \propref{mu0t-tight}.

\begin{cor}
For the policy $(\mathbf{t},\mathbf{y})$ of \defref{determ-ty-def}, define the occupation measures $\{m_{0,t}\}$ by \eqref{mus-t-def}.  Then $\{m_{0.t}\}$ is tight as $t\rightarrow \infty$; that is, for each $\epsilon > 0$, there exists some compact set $\Gamma \subset \RR_+$ and $T < \infty$ such that $m_{0,t}(\Gamma^c) < \epsilon$ for all $t \geq T$.
\end{cor}

Now we get to the point of the example.

\begin{prop}
For the policy $(\mathbf{t},\mathbf{y})$ of \defref{determ-ty-def}, the collection of ordering measures $\{m_{1,t}: t>0\}$ defined by \eqref{determ-mus-t-def} is not tight.
\end{prop}

\begin{proof}
Consider the measure $m_{1,t_{n+1,1}}$ for each $n\in \NN$.  We note that there are $2^{n-1}$ orders of size $0$ at time $t_{n,2^{n-1}+1}<t_{n+1,1}$ when the process $x$ takes value $2^{(n-1)/2}$.  Thus, 
$$m_{1,t_n}(\{(2^{(n-1)/2},2^{(n-1)/2})\}) = \frac{2^{n-1}}{2^n + \frac{2^{n/2}-1}{2^{1/2}-1} - 1} = \frac{1}{2}\cdot \frac{1}{1+\frac{2^{-(n/2)-1}-2^{-n}}{2^{1/2}-1} - 2^{-n}}.$$
Thus for any $\epsilon < \frac{1}{2}$, for any compact set $\Gamma \subset \overline{\cal R} = \{(y,z)\in \RR^2: y\leq z\}$, there is an $n$ sufficiently large so that $(2^{(n-1)/2},2^{(n-1)/2}) \in \Gamma^c$ and $m_{1,t_n}(\{(2^{(n-1)/2},2^{(n-1)/2})\}) > \epsilon$.
\end{proof}

\section{Drifted Brownian Motion Example}
The second example considers a stochastic variation on the deterministic demand model.  The inventory level $X_0$ {\em without orders}\/ is given by the drifted Brownian motion process with
\begin{equation} \label{dbm-dyn}
X_0(t) = W(t) - t, \qquad t\geq 0;
\end{equation}
note that the initial inventory level is $x_0=0$ and the drift and diffusion coefficients are $\mu=-1$ and $\sigma=1$.  The cost functions remain as 
$$c_0(x) = 2|x|, \quad \forall x \in \RR, \qquad \mbox{and}\qquad c_1(y,z) = k_1 + (z-y), \quad \forall (y,z) \in \overline{\cal R}.$$

The ordering policy $(\tau,Y)$ is essentially the same as in the deterministic example, though adapted to the stochatic environment.  The policy runs in cycles, each of which is composed of two phases.  For cycle $i=1,2,3,\ldots$, Phase 1 consists of using the $(0,1)$-ordering policy a total of $2^{i-1}$ times; the length of each sub-cycle is now a random variable having mean $1$.  Phase 2 again involves a single $(0,2^{(i-1)/2})$-ordering policy followed immediately by using the $(2^{(i-1)/2},2^{(i-1)/2})$-ordering policy $2^{i-1}$ times.

The formal description of the ordering policy is now given; as in the deterministic example, we employ a double subscript notation.  

\begin{defn}[{\em The Policy}\/ $(\tau,Y)$] \label{stoch-tauY-def}
For cycle $i=1$, define 
$$\left\{\begin{array}{rcl}
\tau_{1,1} &=& 0, \\ Y_{1,1} &=& 1, \end{array} \right. \qquad \mbox{and} \qquad 
\left\{\begin{array}{rcl}
\tau_{1,2} &=& \inf\{t\geq \tau_{1,1}: X(t) = 0\}, \\ Y_{1,2} &=& 1, \end{array} \right. \qquad 
\left\{\begin{array}{rcl}
\tau_{1,3} &=& \tau_{1,2}, \\ Y_{1,3} &=& 0, \end{array} \right.$$
and for cycle $i = 2,3,4,\ldots$, define the orders in Phase 1 to be 
$$\begin{array}{ll}
\left\{\begin{array}{rcl}
\tau_{i,1} &=& \inf\{t\geq \tau_{i-1,2^i+1}: X(t) = 0\}, \\ Y_{i,1} &=& 1, \end{array} \right. & \rule[-18pt]{0pt}{12pt} \\
\left\{\begin{array}{rcl} 
\tau_{i,j} &=& \inf\{t\geq \tau_{i,j-1}: X(t) = 0\}, \\ Y_{i,j} &=& 1, \end{array} \right. & \quad j=2, \ldots, {2^{i-1}}, \rule[-18pt]{0pt}{12pt} 
\end{array} $$
and the orders in Phase 2 by 
$$\begin{array}{ll}
\left\{\begin{array}{rcl}
\tau_{i,2^{i-1}+1} &=& \inf\{t\geq \tau_{i,2^{i-1}}: X(t) = 0\}, \\ Y_{i,2^{i-1}+1} &=& 2^{(i-1)/2}, \end{array} \right. &  \rule[-18pt]{0pt}{12pt} \\
\left\{\begin{array}{rcl}
\tau_{i,j} &=& \tau_{i,2^{i-1}+1} \\
Y_{i,j} &=& 0, \end{array}\right. & j = 2^{i-1}+2, \ldots, 2^i+1.
\end{array} $$
\end{defn}

The same observation about $0$-size orders as in \remref{no-0-size-order} holds for the stochastic model as well.  The main result of the paper can now be stated. 

\begin{thm}
For the drifted Brownian motion inventory model, let $(\tau,Y)$ be the ordering policy of \defref{stoch-tauY-def}, $X$ be the resulting inventory process, and $\{\mu_{0,t}\}$ and $\{\mu_{1,t}\}$, respectively, be the corresponding average expected occupation and ordering measures defined in \eqref{mus-t-def}.  Then 
\begin{description}
\item[(a)] $J_0(\tau,Y) < \infty$;
\item[(b)] $\{\mu_{0,t}: t > 0\}$ is tight as $t\rightarrow \infty$; and 
\item[(c)] $\{\mu_{1,t}: t > 0\}$ is not tight.
\end{description}
\end{thm}

\begin{proof}
This theorem is proven in pieces.  \propref{lta-holding-costs} shows that the long-term average holding costs are bounded.  In light of this proposition, \propref{mu0t-tight} then establishes the tightness of the average expected occupation measures $\{\mu_{0,t}\}$ as $t\rightarrow \infty$.  Finally, \propref{ordering-measure-results} shows both that the long-term average ordering costs are finite and that the average expected ordering measures $\{\mu_{1,t}\}$ are not tight.
\end{proof}

Our analysis depends on a careful construction of the inventory process $X$ under the ordering policy $(\tau,Y)$ of \defref{stoch-tauY-def}. Independent copies of the diffusion $X_0$ of \eqref{dbm-dyn} are pieced together at the jump times (see, e.g., the appendix of \cite{chri:14} for such a construction) with the implication that the various ordering sub-cycles are independent.  

The initial analysis examines the long-term average cost of the $(\tau,Y)$ policy.  It begins by focusing on the holding costs.  Notice that the only orders which affect the length of cycle $i$ are the $2^{i-1}$ times that the $(0,1)$ policy is used and the one time that the $(0,2^{(i-1)/2})$ policy occurs; the $2^{i-1}$ times that orders of size $0$ are placed do not change the state of the inventory or lengthen the sub-cycles so have no affect on the holding costs.  Thus it is sufficient to restrict the analysis solely to the non-zero orders.

Let ${\cal S}:=\{\sigma_1, \sigma_2,\sigma_3, \ldots\}$ denote the times of the non-zero orders.  Let $i\in \NN$ denote the cycle and $j\in \{1,2,\ldots,2^{i-1}+1\}$ be the number of the non-zero order within cycle $i$.  Then observe that $\sigma_n = \tau_{i,j}$ where $n=2^{i-1}+i+j-2$.  The ensuing computations are simplified by a shift in the index for $j$.  For $i\geq 2$, define $\tau_{i,0} = \tau_{i-1,2^{i-1}+1}$ so that order number ``zero'' of the $i^{\mbox{\footnotesize th}}$ cycle is the last non-zero order, in fact the large order, of cycle $(i-1)$.  

Our first result gives a strong law of large numbers result for the cycle lengths.

\begin{prop} \label{slln}
Let $(\tau,Y)$ be the ordering policy of \defref{stoch-tauY-def} and ${\cal S}$ be the times of the non-zero orders.  Then
\begin{equation} \label{slln-sigma-n}
\lim_{n\rightarrow \infty} \frac{\sigma_n}{n} = 1 \; (a.s.).
\end{equation}
\end{prop} 

\begin{proof}
Define the independent (but not identically distributed) random variables 
$$\beta_0 = \sigma_1 = 0, \quad \mbox{and} \quad \beta_k = \sigma_{k+1} - \sigma_k, \qquad k\in \NN;$$  
thus $\beta_k$ gives the random length of time of the $k^{\mbox{\footnotesize th}}$ inter-order interval.  Then 
$$\sigma_n = \sum_{k=0}^{n-1} \beta_k, \qquad n \in \NN.$$
Thus \eqref{slln-sigma-n} has the representation
\begin{equation} \label{representation}
\lim_{n\rightarrow \infty} \mbox{$\frac{1}{n}$} \sum_{k=0}^{n-1} \beta_k = 1,\; (a.s.).
\end{equation}

Notice that, apart from $\beta_0=0$, $\beta_n$ is either the length of a $(0,1)$ sub-cycle or a $(0,2^{(i-1)/2})$ sub-cycle.  More precisely, for each $i\ge 1$, $\beta_{2^{i}+i-1} $ is the length of the cycle arising from the $(0,2^{(i-1)/2})$ sub-cycle and for $n\neq 2^i-1+i$, $\beta_n$ is the length of a $(0,1)$ sub-cycle.  Using the Laplace transform of the hitting time of a drifted Brownian motion process (see Formula 2.0.1 (p.\;295) of \cite{boro:02}), one can determine that 
$$\begin{array}{lclclcll} 
\EE[\beta_{2^{i-1}+i+j-2}]&=& 1, & \mbox{and} & \mbox{Var}(\beta_{2^{i-1}+i+j-2})&=&1, & \; j=1,2,\ldots, 2^{i-1},\/ i\ge 1, \\
\EE[\beta_{2^{i}+i-1}]&=& 2^{(i-1)/2}, & \mbox{and} & \mbox{Var}(\beta_{2^{i}+i-1})&=& 2^{(i-1)/2}, & \;  i\geq 1.
\end{array}$$
Next observe 
\begin{eqnarray*}
\sum_{n=2}^\infty \frac{\mbox{Var}(\beta_n)}{n^2} &=& 
\sum_{i=2}^\infty \sum_{j=0}^{2^{i-1}}\frac{\mbox{Var}(\beta_{2^{i-1}+i+j-2})}{(2^{i-1}+i+j-2)^2} \\
&\leq& \sum_{i=1}^\infty \sum_{j=1}^{2^{i-1}} \frac{1}{(2^{i-1}+i+j-2)^2} + \sum_{i=1}^\infty \frac{2^{(i-1)/2}}{(2^{i-1}+i-1)^2} \\
&\leq& \sum_{n=1}^\infty \frac{1}{n^2} + \sum_{i=1}^\infty \frac{1}{2^{3(i-1)/2}} < \infty.
\end{eqnarray*} 
By Kolmogorov's Strong Law of Large Numbers (cf.\ Theorem 2 (p.\;389) of \cite{shir:96}), it follows that
$$\lim_{n\rightarrow \infty} \left(\mbox{$\frac{1}{n}$} \sum_{k=1}^n \beta_k - \overline{\mu}_n\right) = 0, \;(a.s.),$$
in which $\overline{\mu}_n = \frac{1}{n} \sum_{k=1}^n \EE[\beta_k]$.  Since $\beta_0=0$, notice that 
$$\mbox{$\frac{1}{n}$} \sum_{k=1}^n \beta_k = \mbox{$\frac{n+1}{n}$} \cdot \left(\mbox{$\frac{1}{n+1}$} \sum_{k=0}^n \beta_k\right)$$
so \eqref{representation} holds if $\overline{\mu}_n \rightarrow 1$ as $n\rightarrow \infty$.

We now analyze the convergence of $\overline{\mu}_n$.  Again for $n\geq 2$, write $n=2^{i-1}+i+j-2$ with $i \ge 2$ and $0\leq j \leq 2^{i-1}$.  Notice that for $1 \leq k \leq i-1$, cycle $k$ contains $2^{k-1}$ sub-cycles generated by $(0,1)$ ordering policies having a mean length of $1$ and a single $(0,2^{(k-1)/2})$ sub-cycle with mean length $2^{(k-1)/2}$ while the partial cycle $i$ has $j$ sub-cycles from $(0,1)$, 
Thus for $n \geq 2$, 
\begin{align*}
\mbox{$\frac{1}{n}$} \sum_{k=1}^n \EE[\beta_k] &
= \frac{1}{2^{i-1}+i+j-2} \left(\sum_{k=1}^{i-1} (2^{k-1} + 2^{(k-1)/2}) + j\right) \\
&= \frac{1}{2^{i-1}+i+j-2} \left(2^{i-1} - 1 +\frac{2^{(i-1)/2}-1}{2^{1/2}-1} + j\right) \\
& = \frac{1 -\frac{1}{2^{i-1}+j} + \frac{1}{2^{1/2}-1} \cdot\frac{2^{(i-1)/2}-1}{2^{i-1}+j} }{1 + \frac{i-2}{2^{i-1}+j}}.
\end{align*} 
Obviously we have $$\lim_{i\to\infty}\frac{1}{2^{i-1}+j}=\lim_{i\to\infty}\frac{2^{(i-1)/2}-1}{2^{i-1}+j} =\lim_{i\to\infty} \frac{i-2}{2^{i-1}+j} =0 \text{ for each }j =0, 1, \dots, 2^{i-1}.$$ 
Therefore it follows that as $n\to\infty$ (and hence   $i\rightarrow \infty$),  $\frac{1}{n} \sum_{k=1}^n \EE[\beta_k] $ converges to 1.
\end{proof}

We now establish a variant of the elementary renewal theorem; the fact that the cycles are not identically distributed means that we cannot simply apply the theorem.

\begin{prop} \label{elementary-renewal}
For $t \geq 0$, let $\displaystyle = \sum_{i=1}^\infty I_{\{\sigma_i \leq t\}} = \max\{n: \sigma_n \leq t\}$ be  the number of orders of positive size by time $t$.  Then 
\begin{description}
\item[(a)] $\displaystyle \lim_{t\rightarrow \infty} \frac{N(t)}{t} = 1, \; (a.s.)$; and
\item[(b)] $\displaystyle \lim_{t\rightarrow \infty} \frac{\EE[N(t)]}{t} = 1$.
\end{description}
\end{prop}

\begin{proof}
First by the definition of $N(t)$, $\sigma_{N(t)} \leq t < \sigma_{N(t)+1}$ for each $t \geq 0$.  Thus for each $t \geq 0$,
$$\frac{\sigma_{N(t)}}{N(t)} \leq \frac{t}{N(t)} < \frac{\sigma_{N(t)+1}}{N(t)} = \frac{N(t)+1}{N(t)} \cdot \frac{\sigma_{N(t)+1}}{N(t)+1}.$$
As $t\rightarrow \infty$, \propref{slln} implies first that $N(t) \rightarrow \infty\; (a.s.)$ and then establishes (a).

In order to prove (b), it is necessary to show that the collection $\{\frac{N(t)}{t}: t\geq 1\}$ is uniformly integrable.  By Lemma~3 (p.\ 190) of \cite{shir:96}, it suffices to show that $\sup_{t\geq 1} \EE[(\frac{N(t)}{t})^2] < \infty$.  To establish this result, we compare the counting process $N$ with a similar renewal process $\wdt N$ connected to a different policy.  

Define the ordering policy $(\wdt\tau,\wdt Y)$ which always uses the $(0,1)$ policy and let $\{\wdt{\sigma}_n: n\in\NN\}$ denote the ordering times.  Define the renewal process $\wdt N$ by  
$$\wdt N(t) = \max\{n: \wdt\sigma_n \leq t\} = \sum_{i=1}^\infty I_{\{\wdt\sigma_i \leq t\}}.$$  
It then follows that $N(t) \leq \wdt N(t)$ for each $t \geq 0$ and hence
$$\EE\left[\left(\frac{N(t)}{t}\right)^2\right] \leq \EE\left[\left(\frac{\wdt N(t)}{t}\right)^2\right].$$
Using a standard renewal argument (see, e.g., the proof of Theorem~5.5.2 (pp.\;143,144) of \cite{chun:01}), it follows that the right-hand side is uniformly bounded for $t \geq 1$, establishing the uniform integrability of $\{\frac{N(t)}{t}: t \geq 1\}$ and hence (b).  
\end{proof}

The next step on the way to showing $J_0(\tau,Y) < \infty$ is to analyze the holding costs over a single cycle.  Observe that $X(\sigma_i)$ is either $1$ or $2^{(k-1)/2}$ for some $k$; to simplify notation, let $z$ represent either value.  Next, $X(t) = z - (t-\sigma_i) + W(t-\sigma_i)$ for $t \in [\sigma_i,\sigma_{i+1})$ since no orders are placed on the interval $(\sigma_i,\sigma_{i+1})$ and, by the definition of $\sigma_{i+1}$, $z - (\sigma_{i+1}-\sigma_i) + W(\sigma_{i+1}-\sigma_i) = 0$.  Again to simplify notation, make the change of time $s=t-\sigma_i$ and define $\tau=\sigma_{i+1}-\sigma_i$.  Thus, $X$ satisfies $X(s) = z - s + W(s)$ for $s\in [0,\tau)$.  


Define 
\begin{equation} \label{cycle-holding-cost}
L_\tau = \int_0^\tau c_0(X(s))\, ds.
\end{equation}
Then by Proposition~2.6 of \cite{helm:16}, it follows that $\EE[L_\tau] = z^2 + z$.  In applying this result, one uses the fact that the function $g_0$ of that paper is computed to be $g_0(x) = x^2 + x$ for this specific example.

We now establish a result similar to the law of large numbers for the holding costs.  

\begin{prop} \label{cycle-cost-asymptotics}
Let $(\tau,Y)$ be given by \defref{stoch-tauY-def} for the drifted Brownian motion process, $X$ be the inventory process and denote the non-zero ordering times by the collection ${\cal S}$.  Then
\begin{equation}
\label{eq-limsup-holding-cost}
 \limsup_{n\rightarrow \infty} \mbox{$\frac{1}{n}$} \sum_{k=1}^{n} \EE\left[\int_{\sigma_{k}}^{\sigma_{k+1}} c_0(X(s))\, ds\right] \le  3.
\end{equation}
\end{prop}

\begin{proof} For simplicity of notation, denote $L_{k} : = \int_{\sigma_{k}}^{\sigma_{k+1}} c_0(X(s))\, ds $ for $k =1, 2, \dots$
As in the proof of \propref{slln}, write the index $n$ as $n=2^{i-1}+i+j-2$ in which $i=2,3,\ldots$ and $j=0,1,\ldots,2^{i-1}$.  Recall, the large orders have indices with $j=0$ and hence $n = 2^{i-1} + i -2 $ for $i \ge 2$ so the formula for $\EE[L_\tau]$ above establishes that 
$$\begin{array}{rccl}
\EE[L_{2^{i-1}+i+j-2}] &=& 2, & \quad j=1,\ldots,2^{i-1},\; i \in \NN, \\
\EE[L_{2^{i}+i-1}] &=& 2^{i-1} + 2^{(i-1)/2}, & \quad  i\in \NN.
\end{array}$$
Consequently, for $n\geq 2$, 
\begin{align*}
 \mbox{$\frac1n$} \sum_{\ell=1}^{n} \EE[L_{\ell}] 
 &  =  \frac{1}{2^{i-1}+i+j-2}   \left(\sum_{k=1}^{i-1} (2^{k-1}\cdot 2 + 2^{k-1} + 2^{(k-1)/2}) +j \cdot 2\right)  \\
 & =  \frac{1}{2^{i-1}+i+j-2} \left(3(2^{i-1} -1) +  \frac{2^{(i-1)/2}-1}{2^{1/2} -1} + 2 j \right) \\
 & \le   \frac{3 ( 2^{i-1}  +j ) + \frac{2^{(i-1)/2}-1}{2^{1/2} -1}-3}{2^{i-1}+i+j-2} \\
 & = \frac{3 + \frac{1}{2^{1/2}-1} \cdot\frac{2^{(i-1)/2}-1}{2^{i-1}+ j} - \frac{3}{2^{i-1}+ j}} {1 + \frac{i-2}{ 2^{i-1} + j}}.
\end{align*} Using similar computations as those in the end of   the proof of Proposition \ref{slln}, we see immediately that this ratio converges to 3 as $n\to \infty $ and hence $i \to \infty$. This gives \eqref{eq-limsup-holding-cost} as desired. 
\end{proof}

We now parlay the asymptotics relative to cycles to verify that the long-term average holding costs related to $(\tau,Y)$ are finite.
\begin{prop} \label{lta-holding-costs}
Let $(\tau,Y)$ be given by \defref{stoch-tauY-def} for the drifted Brownian motion process having underlying diffusion given by \eqref{dbm-dyn} and let $X$ satisfy \eqref{controlled-dyn} .  Then
$$\limsup_{t\rightarrow \infty} \mbox{$\frac{1}{t}$} \EE\left[\int_0^t c_0(X(s))\, ds\right] \leq 3.$$
\end{prop}

\begin{proof}
Again, denote the non-zero ordering times by ${\cal S}$ and observe that the orders of size $0$ do not affect the cycle lengths.  Again, for each $t \geq 0$, recall $N(t) = \max\{n: \sigma_n \leq t\}$ so that $\sigma_{N(t)} \leq t < \sigma_{N(t)+1}$.  Thus for positive $t$,
\begin{eqnarray*}
\mbox{$\frac{1}{t}$} \int_0^t c_0(X(s))\, ds &\leq& \mbox{$\frac{1}{t}$} \int_0^{\sigma_{N(t)+1}} c_0(X(s))\, ds  
 \leq  \mbox{$\frac{1}{t}$} \int_0^{\sigma_{N(t)+2}} c_0(X(s))\, ds =  \mbox{$\frac{1}{t}$} \sum_{j=1}^{N(t)+1} L_{j}, 
\end{eqnarray*}
where as in the proof of \propref{cycle-cost-asymptotics}, we employed the notation $L_j = \int_{\sigma_j}^{\sigma_{j+1}} c_0(X(s))\, ds$ for each $j \in \NN$.  Noting that $\{N(t)=0\} = \emptyset$ for each $t$, consider the expectation of the right-hand term above:
\begin{eqnarray} \label{approx-wald} \nonumber
\EE\left[\sum_{j=1}^{N(t)+1} L_j\right] &=& \EE\left[\sum_{k=1}^\infty I_{\{N(t)=k\}}\sum_{j=1}^{k+1} L_j\right] \\
&=& \EE\left[\sum_{j=1}^\infty L_j\sum_{k=j-1}^\infty I_{\{N(t)=k\}}\right] = \sum_{j=1}^\infty \EE\left[L_j I_{\{N(t) \geq j-1\}}\right].
\end{eqnarray}
Observe that $\{N(t) \geq j-1\} = \{N(t)<j-1\}^c = \{\sigma_{j-1} > t\}^c$ is independent of the process over the interval $[\sigma_j,\sigma_{j+1})$.  Therefore the right-hand side of \eqref{approx-wald} yields
\begin{eqnarray*}
\sum_{j=1}^\infty \EE\left[L_j I_{\{N(t) \geq j-1\}}\right] &=& \sum_{j=1}^\infty \EE[L_j] \EE[I_{\{N(t) \geq j-1\}}] \\
&=& \sum_{j=1}^\infty \EE[L_j] \left(\sum_{k=j-1}^\infty \PP(N(t)=k)\right) \\
&=& \sum_{k=1}^\infty \sum_{j=1}^{k+1} \EE[L_j] \PP(N(t)=k) \\
\end{eqnarray*}
By \propref{cycle-cost-asymptotics}, $\limsup_{k\rightarrow \infty} \frac{1}{k}\sum_{j=1}^{k} \EE[L_j] \le 3$ so for any $\epsilon > 0$ there exists some $K_0 < \infty$ such that $\frac{1}{k+1} \sum_{j=1}^{k+1} \EE[L_j] <   3+ \epsilon$ for all $k \geq K_0$.  Thus
 \begin{align*}
&\sum_{k=1}^\infty \sum_{j=1}^{k+1} \EE[L_j] \PP(N(t)=k)  \\
 &\ \ =  \sum_{k=1}^{K_0-1}   \sum_{j=1}^{k+1} \EE[L_j] \PP(N(t)=k) +  \sum_{k=K_0}^\infty \left(\mbox{$\frac{1}{k+1}$} \sum_{j=1}^{k+1} \EE[L_j]\right) (k+1) \PP(N(t)=k) \\
&\ \ \leq \sum_{k=1}^{K_0-1}  \sum_{j=1}^{k+1} \EE[L_j]  \PP(N(t)=k)  +\sum_{k=K_0}^\infty (3 + \epsilon) (k+1) \PP(N(t)=k) \\
&\ \ \leq \sum_{k=1}^{K_0-1} \sum_{j=1}^{k+1} \EE[L_j] + (3 + \epsilon)\EE[N(t)] +3 + \epsilon.
\end{align*}Since the first and last summands are constant, combining these upper bounds, dividing by $t$ and using \propref{elementary-renewal} yields
$$\limsup_{t\rightarrow \infty} \mbox{$\frac{1}{t}$}\EE\left[\int_0^t c_0(X(s))\,ds\right] \leq \lim_{t\rightarrow \infty} \frac{(3 + \epsilon)\EE[N(t)]}{t} =3 + \epsilon.$$
The result now follows since $\epsilon> 0$ is arbitrary.\end{proof}

It is now easy to show that $\{\mu_{0,t}\}$ is tight as $t\rightarrow \infty$.  

\begin{prop} \label{mu0t-tight}
Let $(\tau,Y)$ be given by \defref{stoch-tauY-def} for the drifted Brownian motion inventory model and for $t > 0$, define $\mu_{0,t}$ by \eqref{mus-t-def}.  Then $\{\mu_{0,t}: t > 1\}$ is tight as $t\rightarrow \infty$.
\end{prop}

\begin{proof}
Choose $\epsilon > 0$ arbitrarily and pick $M$ such that $M > \frac{3}{\epsilon}+1$.  Let $T$ be such that for all $t \geq T$,
$$\mbox{$\frac{1}{t}$} \EE\left[\int_0^t c_0(X(s))\, ds\right] < 3 + \epsilon.$$
Since $c_0(x) = 2|x|$ is inf-compact, define the compact set $\Gamma = [-\frac{M}{2},\frac{M}{2}] = \{x: c_0(x) \leq M\}$.  Then for each $t>T$, 
$$ \mu_{0,t}(\Gamma^c) \leq \int_{\Gamma^c} \frac{c_0(x)}{M}\, \mu_{0,t}(dx) \leq \frac{1}{M}\int c_0(x)\, \mu_{0,t}(dx)  \leq \frac{3 + \epsilon}{M}  <  \epsilon.$$
Since $\epsilon$ is arbitrary, it follows that $\{\mu_{0,t}\}$ is tight as $t\rightarrow \infty$.
\end{proof}

The final task is to verify that the long-term average ordering costs are finite and that $\{\mu_{1,t}\}$ is not tight as $t\rightarrow \infty$.  The next proposition addresses both of these concerns since the analysis is very similar.

\begin{prop} \label{ordering-measure-results}
Let $(\tau,Y)$ be given by \defref{stoch-tauY-def} for the drifted Brownian motion inventory model.  For $t > 0$, define $\mu_{1,t}$ by \eqref{mus-t-def}.  Then 
\begin{equation}
\label{eq-limsup-order-cost}
 \limsup_{t\rightarrow \infty} \int c_1(y,z)\, \mu_{1,t}(dy\times dz) \le 3k_1+2 
\end{equation}
and $\{\mu_{1,t}: t > 1\}$ is not tight. 
\end{prop}

\begin{proof}
We first address the lack of tightness for $\{\mu_{1,t}\}$.  Let $\Gamma\subset \overline{\cal R}$ be any compact set.  Then there exists some $N_0$ such that for all $i \geq N_0$, $(2^{(i-1)/2},2^{(i-1)/2}) \in \Gamma^c$.  

Again, denote the times of non-zero orders by ${\cal S} = \{\sigma_n: n\in \NN\}$ and for $n\geq 2$, write $n=2^{i-1}+i+j-2$ with $i \ge 2$ and $j=0,\ldots,2^{i-1}$.  Recall, the ``large'' order of size $2^{(i-1)/2}$ occurs at time $\sigma_{2^i+i-1}$.  Under policy $(\tau,Y)$, there are a further $2^{i-1}$ orders of size $0$ at time $\sigma_{2^i+i-1}$; denote this common time of ordering by $\wdt\sigma_{2^i+i-1,j}$ for $j=1,\ldots,2^{i-1}$ for each for the $0$-size orders.  Thus
\begin{eqnarray*}
\mu_{1,t}(\Gamma^c) &=& \mbox{$\frac{1}{t}$} \EE\left[\sum_{n=1}^\infty I_{\{\sigma_n \leq t\}} I_{\Gamma^c}(X(\sigma_n-),X(\sigma_n)) \right. \\
& & \left. \qquad +\; \sum_{i=1}^\infty \sum_{j=1}^{2^{i-1}} I_{\{\wdt\sigma_{2^i+i-1,j}\leq t\}} I_{\Gamma^c}(X(\wdt\sigma_{2^i+i-1,j}-),X(\wdt\sigma_{2^i+i-1,j})) \right] \\
&\geq& \mbox{$\frac{1}{t}$} \EE\left[\sum_{i=N_0}^\infty \sum_{j=1}^{2^{i-1}} I_{\{\wdt\sigma_{2^i+i-1,j}\leq t\}} I_{\Gamma^c}(2^{(i-1)/2},2^{(i-1)/2}) \right] \\
&=& \mbox{$\frac{1}{t}$} \EE\left[\sum_{i=N_0}^\infty 2^{i-1} I_{\{\sigma_{2^i+i-1}\leq t\}} \right].
\end{eqnarray*}
For each $t \geq 0$, define the processes $I$ and $J$ such that $N(t) = 2^{I(t)-1}+I(t)+J(t)-2$ in which $I(t)$ denotes the cycle in which order $N(t)$ occurs and $0 \leq J(t) \leq 2^{I(t)-1}$.  Since $N(t)$ is the number of non-zero orders placed by time $t\geq 0$, order number $N(t)$ is the $J(t)^{\mbox{\footnotesize th}}$ order within cycle $I(t)$; again $J(t)=0$ corresponds to the large order of the previous cycle.  Using the processes $I$ and $J$, it follows that for $0 \leq J(t) \leq 2^{I(t)-1}$,
\begin{eqnarray} \label{mu1-lb} \nonumber 
\mu_{1,t}(\Gamma^c) &\geq& \mbox{$\frac{1}{t}$} \EE\left[\sum_{i=N_0}^\infty 2^{i-1} I_{\{\sigma_{2^i+i-1}\leq t\}} \right] \geq \mbox{$\frac{1}{t}$} \EE\left[\sum_{\ell=2^{N_0}+N_0-1}^{I(t)-1} 2^{\ell-1}\right] \\ \nonumber
&=& \mbox{$\frac{1}{t}$} \EE\left[\sum_{\ell=1}^{I(t)-1} 2^{\ell-1} - \sum_{\ell=1}^{2^{N_0}+N_0-2} 2^{\ell-1}\right] \\
&=& \mbox{$\frac{1}{t}$} \left(\EE\left[2^{I(t)-1}\right] - 2^{2^{N_0}+N_0-2}\right).
\end{eqnarray}
 By \lemref{elementary-renewal}, $N(t) \rightarrow \infty\; (a.s.)$ as $t\rightarrow \infty$ so $I(t) \rightarrow \infty$ as well.  Thus the asymptotics of $\mu_{1,t}(\Gamma^c)$ is determined by the asymptotics of the first summand above. 

We next determine bounds on $I(t)$.  Since $J(t) \leq 2^{I(t)-1}$ and $N(t) = 2^{I(t)-1}+I(t)+J(t)-2$, we have $N(t) \leq 2^{I(t)}+I(t) -2$ and hence 
\begin{equation} \label{bd1}
2^{I(t)-1} \geq \mbox{$\frac{1}{2}$} (N(t) - I(t) +2).
\end{equation}
Since $I(t)\ge 1 $ and  $J(t) \ge  0$, $N(t) \ge  2^{I(t)-1}-1$ so 
\begin{equation} \label{I-minus-one-lb}
I(t)-1 < \log_2(N(t)+1) \leq \log_2(N(t)+N(t)) = \mbox{$\log_2(\frac{N(t)\cdot 2t}{t}) = \log_2(\frac{N(t)}{t}) + \log_2(2t)$}.
\end{equation}
Using this estimate in \eqref{bd1} yields
$$2^{I(t)-1} > \mbox{$\frac{1}{2}$} (N(t) - \log_2(\mbox{$\frac{N(t)}{t}$}) - \log_2(2t)+ 1).$$
Employing this lower bound in \eqref{mu1-lb} and Jensen's inequality on the second summand, we have
$$\mu_{1,t}(\Gamma^c) \geq \mbox{$\frac{1}{2}$} \cdot \left(\frac{\EE[N(t)]}{t} - \frac{\log_2(\EE[\frac{N(t)}{t}])}{t} - \frac{\log_2(2t)}{t}\right) \stackrel{t\rightarrow \infty}{\longrightarrow} \mbox{$\frac{1}{2}$};$$
note that we have also used \propref{elementary-renewal} on the second summand.  Therefore for any $\epsilon < \frac{1}{2}$, $\mu_{1,t}(
\Gamma^c) > \epsilon$ for all $t$ sufficiently large.  Hence $\{\mu_{1,t}\}$ is not tight as $t\rightarrow \infty$.

Consider now the total ordering costs by time $t > 0$.  First, denote $(\tau,Y) = \{(\tau_k,Y_k): k \in \NN\}$ to capture all of the orders.  Next, denote the non-zero orders by ${\cal S}$ and as above, for each $i \in \NN$ and $j=1,\ldots, 2^{i-1}$, let $\wdt\sigma_{2^i+i-1,j} = \sigma_{2^i+i-1}$ be the common time of the $0$-size orders in cycle $i$.  

Since $t$ is finite and $0 \leq J(t) \leq 2^{I(t)-1} $, 
\begin{align*}
\lefteqn{\EE\left[\sum_{k=1}^\infty I_{\{\tau_k\leq t\}} c_1(X(\tau_k-),X(\tau_k))\right]} \\
   & =  \EE\left[\sum_{i=1}^{I(t)-1} \left[2^{i-1}(k_1+1)  + (k_1 + 2^{(i-1)/2}) + 2^{i-1} k_{1}\right] + J(t) (k_{1}+1)\right]\\ 
    & \le     \EE\left[\left( (2k_1+1) (2^{I(t)-1}-1) + \frac{2^{(I(t)-1)/2}-1}{2^{1/2}-1} + k_1 (I(t)-1)\right)+ 2^{I(t) -1} (k_{1}+1)\right]\\
&=  \EE\left[(3k_1+2) 2^{I(t)-1}  + \frac{2^{(I(t)-1)/2}-1}{2^{1/2}-1} + k_1 I(t)  -3 k_{1} -1 \right].
\end{align*}
Since $I(t) \ge 1$ and $ J(t) \geq 0$, it follows that $2^{I(t)-1}    \le N(t) + 1 $ and so $2^{(I(t)-1) /2} \le (N(t)+1)^{1/2}$.  Using \eqref{I-minus-one-lb}, we also have $I(t) <  \log_{2}(\frac{N(t)}{t}) + \log_{2}(2t) + 1$. 
Then it follows from Jensen's inequality  that \begin{align*}
\EE&\left[\sum_{k=1}^\infty I_{\{\tau_k\leq t\}} c_1(X(\tau_k-),X(\tau_k))\right] \\
      & \le  (3k_{1} +2) \EE[N(t)]    + \frac{     \EE[(N(t)+1)^{1/2}]-1}{2^{1/2}-1}  + k_{1} \EE\left[ \log_{2}(N(t)/t ) + \log_{2}(2t) + 1\right] + 1 \\ 
      & \le   (3k_{1} +2) \EE[N(t)]    + \frac{   (\EE[N(t)+1])^{1/2} - 1} {2^{1/2}-1}    + k_{1}\big [ \log_{2} (\EE[N(t)/t] ) + \log_{2}(2t) + 1\big] + 1. 
\end{align*}
Now divide both sides by $t$, and then send $t\to \infty$, obtaining \eqref{eq-limsup-order-cost}  from Proposition \ref{elementary-renewal}.
\end{proof}

\begin{rem}[Final Comments]
Under further analysis, one can obtain more precise results about the costs related to the ordering policy $(\tau,Y)$ of \defref{stoch-tauY-def}.  As in the analysis of \propref{cycle-cost-asymptotics}, a lower bound on the limit inferior of the Ces\`{a}ro mean of the expected cycle costs can be shown to be $\frac{5}{2}$.  With more extensive calculations including the variance of the holding costs per cycle, these bounds can be shown to be tight and moreover that 
$$\liminf_{t\rightarrow \infty} \mbox{$\frac{1}{t}$} \int_0^t c_0(X(s))\, ds = \mbox{$\frac{5}{2}$} \; (a.s.) \quad \mbox{and} \quad \limsup_{t\rightarrow \infty} \mbox{$\frac{1}{t}$} \int_0^t c_0(X(s))\, ds = 3 \; (a.s.).$$
We also note that \propref{mu0t-tight} is proven for general inventory models in \cite{helm:17}; for completeness of the paper, we have included \propref{mu0t-tight} for this example.  
\end{rem}

\end{document}